\documentclass[11pt]{amsart}
\usepackage{graphicx}

\theoremstyle{plain}

\begin{document}


\theoremstyle{plain}
\newtheorem{theorem}{Theorem} [section]
\newtheorem{corollary}[theorem]{Corollary}
\newtheorem{lemma}[theorem]{Lemma}
\newtheorem{proposition}[theorem]{Proposition}


\theoremstyle{definition}
\newtheorem{definition}[theorem]{Definition}
\theoremstyle{remark}
\newtheorem{remark}[theorem]{Remark}

\numberwithin{theorem}{section}
\numberwithin{equation}{section}
\numberwithin{figure}{section}

\title[On the ABP estimate for the infinity laplacian]{On the Aleksandrov-Bakelman-Pucci estimate for the infinity Laplacian}

\thanks{F.C. partially supported by a MEC-Fulbright fellowship and Spanish project  MTM2010-18128, MICINN.
A.D.C. supported by CMUC/FCT and project UTAustin/MAT/0035/2008. 
  D.M. partially supported by NFS grant DMS-0545984.}

\author[F. Charro]{Fernando Charro}

\address{Department of Mathematics, The University of Texas, 1 University Station C1200, Austin, TX 78712}

\email{fcharro@math.utexas.edu and maximo@math.utexas.edu}

\author[G. De Philippis]{Guido De Philippis}
\address{Scuola Normale Superiore,
p.za dei Cavalieri 7, I-56126 Pisa, Italy}
\email{guido.dephilippis@sns.it}

\author[A. Di Castro]{Agnese Di Castro}
\address{CMUC, Department of Mathematics, University of Coimbra 3001-454 Coimbra, Portugal}
\email{dicastro@mat.uc.pt}

\author[D. M\'aximo]{Davi M\'aximo}


\date{}

\maketitle

\begin{abstract}
We prove $L^\infty$ bounds  and estimates of the modulus of continuity of solutions to the Poisson problem for the normalized infinity and $p$-Laplacian, namely
\[
-\Delta_p^N u=f\qquad\text{for $n<p\leq\infty$.}
\]
 We are able to provide a stable family of results depending continuously on the parameter $p$. We also prove the failure of the classical Alexandrov-Bakelman-Pucci estimate for the normalized infinity Laplacian and propose alternate estimates.
\end{abstract}

\section{Introduction}\label{section.intro}

In this paper we investigate the validity of Alexan\-drov-Bakelman-Pucci estimates (ABP for short) for the  infinity Laplacian,
\begin{equation}\label{inflap}
\Delta_\infty u:=\langle D^2u\nabla u, \nabla u\rangle 
\end{equation}
and its normalized version
\begin{equation}\label{inflapnorm}
\Delta_\infty^N u:=\Big\langle D^2u \frac{\nabla u}{|\nabla u|}, \frac{\nabla u}{|\nabla u|}\Big\rangle
\end{equation}
 as well as    $L^\infty$  and $\mathcal{C}^{0,\alpha}$  estimates, uniform in $p$,   for the solutions of the Poisson problem for the normalized $p$-Laplacian
 \begin{equation}\label{p.finite.xxx}
\Delta_{p}^N u:= \frac1p\,|\nabla u|^{2-p}\,\text{div}\big(|\nabla u|^{p-2}\nabla u\big).
 \end{equation}
 
 These type of normalized operators have recently received great attention, mainly since they have an interpretation in terms of random Tug-of-War games, see \cite{PSSW, PSSW2}.

As an application of these estimates, in Section \ref{section.limits} we  prove convergence as $p\to\infty$ of solutions to 
\begin{equation}\label{nespresso}
-\Delta_p^N u=f,
\end{equation}
to solutions of
\begin{equation}\label{coffa}
-\Delta_\infty^N u=f.
\end{equation}
It is interesting to compare this result with the results in \cite{Bha-DiBe-Man}. An important feature of this limit process is the lack of variational structure of problem \eqref{nespresso} that yields complications in the proof of the uniform convergence of the solutions (for instance, Morrey's estimates, are not available). It is in this context that our ABP estimate provides a stable $L^\infty$ bound that can be used in combination with our stable regularity results to prove convergence.

\medskip

The Alexandrov-Bakelman-Pucci maximum principle is well-known since the decade of 1960 in the context of linear uniformly  elliptic  equations. It can be stated as follows,
\begin{equation}\label{ABP.justdoit}
\sup_{\Omega} u\leq  \sup_{\partial\Omega} u + C(n,\lambda,\Lambda)\,{\rm diam}\,(\Omega)\|f\|_{L^n(\Omega)},
\end{equation}
for $f$ the right-hand side of the equation and $0<\lambda\leq\Lambda$ the ellipticity constants.

  These estimates have been generalized in several directions, including uniformly elliptic fully nonlinear equations where they are the central tool in the proof of the Krylov-Safonov Harnack inequality and regularity theory (see \cite{CC} and the references therein).

In Theorem \ref{ABP.plapnorm.norm.n} we recall the classical ABP estimate for solutions of \eqref{nespresso} (see  \cite{ACP,CC, DFQ,Imbert} for different proofs of this result) in order to stress that the constant blows up as $p\to\infty$. This fact is not merely technical, actually we are able to prove that the classical ABP estimate fails to hold for the normalized infinity Laplacian (see Section \ref{section.examples}). An open question is if an estimate of the form \eqref{ABP.justdoit} holds for some other integral norm $\|f\|_{L^p(\Omega)}$ with $p>n$.

The failure of the classical ABP estimate reflects the very high degeneracy of the equation as the normalized infinity Laplacian only controls  the second derivative in the direction of the gradient.  Remarkably, this is still enough to get an estimate of $\sup u$ of the type,
\[
\big(\sup_{\Omega }u-\sup_{\partial\Omega}u\big)^2\leq C\, {\rm diam}
(\Omega)^2 \int _{\sup_{\partial\Omega}u}^{\sup_{\Omega}u}
\|f\|_{L^\infty(\{u=r\})}\,dr,
\]
since the second derivatives in the directions tangential to the level sets can be controlled in average using the Gauss-Bonnet theorem, see Section \ref{section.ABP.inflap} for a more precise statement. See also Section \ref{section.examples} for an example in which this estimate is sharper than the estimate with a plain $\|f\|_{L^\infty(\Omega)}$ in the right-hand side.

With these ideas we are able to provide a stable family of estimates depending continuously on the parameter $p\in(1,\infty]$. Incidentally, our ABP-type estimates turn out to be independent of the dimension $n$.  Maybe this could have some interest in the context of Tug-of-War games in infinity dimensional spaces.

Let us now briefly discuss the $\mathcal{C}^{0,\alpha}$ estimates. The well-known Krylov-Safonov estimates in \cite{CC} apply to the normalized $p$-Laplacian whenever $p<\infty$, but degenerate as $p\to\infty$ as they depend upon the ratio between the ellipticity constants (see also \cite[Section 9.7 and 9.8]{GT}), which in this case is $p-1$  and  blows-up as $p\to\infty$.

We prove H\"older estimates for solutions of the normalized $p$-Laplacian as well as Lipschitz estimates for the normalized $\infty$-Laplacian. The main interest of this estimates is that they are stable in $p$, so that the whole range $n<p\leq\infty$ can be treated in a coherent way  with all the parameters involved varying continuously.

The ideas behind our proof are somehow reminiscent of the comparison with cones property in \cite{Crandall.Evans.Gariepy} for the homogeneous infinity Laplace equation and the comparison with polar quadratic polynomials introduced in \cite{armstrong.smart, LuWangCPDE} to study the non-homogeneous normalized $\infty$-Laplace equation.
Actually, our Lipschitz estimates in the case of the infinity Laplacian are closely related to \cite[Lemma 5.6]{armstrong.smart}.

Although we shall not address this issue here, let us mention that existence of solutions to  \eqref{nespresso} and \eqref{coffa} is a nontrivial question as the normalized infinity and $p$-Laplacian are non-variational (in contrast to the regular $p$-Laplacian) and discontinuous operators (see Section \ref{section.prelim}). 

Existence of viscosity solutions to \eqref{coffa} with Dirichlet boundary conditions has been proved using game-theoretic arguments, finite difference methods, and PDE techniques, see \cite{armstrong.smart, LuWangCPDE, PSSW}. Moreover, solutions to the Dirichlet problem can be characterized using comparison with quadratic functions.  

Uniqueness of  solutions to the Dirichlet problem is known if the right-hand side $f\in\mathcal{C}(\Omega)\cap L^\infty(\Omega)$ is either 0, or has constant sign (see \cite{armstrong.smart, Jensen93,LuWangCPDE}). When this conditions are not met, interesting non-uniqueness phenomena can happen, see \cite{armstrong.smart} for a detailed discussion. 

About \eqref{nespresso}, we first notice that when the right-hand side is 0, both the normalized and regular $p$-Laplacian coincide. For a non-trivial right-hand side, existence of viscosity solutions can be addressed with a game-theoretic approach (see \cite{PSSW2}). However, to the best of our knowledge, conditions that guarantee uniqueness are not  well understood.

The paper is organized as follows. In Section \ref{section.prelim} we introduce some notation and preliminaries. In Sections \ref{section.ABP.inflap}  and \ref{section.ABP.plap} we prove  the ABP-type estimates. Then, in Section \ref{section.modulus.cont} we provide stable estimates of the modulus of continuity of the solutions in the case $n<p\leq\infty$. Next, in Section \ref{section.limits} we justify the limits as $p\to\infty$ and 
finally, in Section \ref{section.examples} we provide examples showing the failure of an estimate of the form \eqref{ABP.justdoit} and analyze the sharpness of our ABP-type estimate.

\section{Notation and preliminaries}\label{section.prelim}
In this section we are going to state notations and recall some facts about the operators we are going to use in the sequel.

In what follows, $\Omega\subset\mathbb{R}^n$, $n\geq2$, will be an open and bounded set. $B_r(x)$ is the open ball of radius $r$ and center $x$. We shall denote the diameter of a domain by $d$ or ${\rm diam} (\Omega)$. We write $1_A(x)$ for the indicator function of a set $A$.
Given two vectors $\xi,\eta\in\mathbb{R}^n$ we denote their scalar product as $\langle\xi,\eta\rangle=\sum_{i=1}^n\xi_i\eta_i$ and their tensor product as the matrix $\xi\otimes\eta=(\xi_i\eta_j)_{1\leq i,j\leq n}$. $S^n$ is the space of symmetric $n\times n$ matrices.

$\mathcal{H}^{n-1}$ stands for the $n-1$ dimensional Haussdorff measure, which for $\mathcal{C}^1$ manifolds coincide with the classical surface measure. Wherever not specified otherwise, ``a.e." will refer to the Lebesgue measure.

Given a function $u$, we shall denote $u^+=\max\{0,u\}$ and $u^-=\max\{0,-u\}$ so that $u=u^+-u^-$. We shall say that a function is twice differentiable at a point $x_0\in\Omega$ if 
\[
u(x)=u(x_0)+\langle\xi,(x-x_0)\rangle+\frac12\langle A(x-x_0),(x-x_0)\rangle +o(|x-x_0|^2)
\]
as $|x- x_0|\to0$ for some $\xi\in\mathbb{R}^n$ and  $A\in S^n$, which are easily seen to be unique and will be denoted by $\nabla u(x_0)$ and $D^2 u(x_0)$ respectively.

A function $u$ is semiconvex  if $u(x)+C|x|^2$ is convex for some constant $C\in\mathbb{R}$. Geometrically, this means that the graph of $u$ can be touched from below by a paraboloid of the type $a+\langle b,x\rangle-C|x|^2$ at every point.

\medskip

The infinity Laplacian and $p$-Laplacian for $1<p<\infty$, as well as its normalized versions have already been defined in the introduction, see \eqref{inflap}, \eqref{inflapnorm}, and \eqref{p.finite.xxx}. For the record, let us recall here different expressions for the variational $p$-Laplacian, $1<p<\infty$,
\begin{equation}\label{plap}
\begin{split}
\Delta_{p}u&:=
\text{div}\big(|\nabla u|^{p-2}\nabla u\big)\\
&=|\nabla u|^{p-2}\cdot
\textnormal{trace}\left[\left(I+(p-2)\frac{\nabla  u\otimes\nabla  u}{|\nabla  u|^2}\right)D^2 u\right]
\end{split}
\end{equation}
and its normalized version
\begin{equation}\label{plapnorm}
\begin{split}
\Delta_p^N u&=
\frac{1}{p}\,\textnormal{trace}\left[\left(I+(p-2)\frac{\nabla  u\otimes\nabla  u}{|\nabla  u|^2}\right)D^2 u\right]\\
&=\frac{1}{p}\,\Delta u+\frac{p-2}{p}\,\Delta_{\infty}^{N}u\\
&=\frac{1}{p}\,|\nabla u|\,\text{div}\Big(\frac{\nabla u}{|\nabla u|}\Big)+\frac{p-1}{p}\,\Delta_{\infty}^{N}u\\
&=\frac{1}{p}\,\Delta_{1}^{N}u+\frac{p-1}{p}\,\Delta_{\infty}^{N}u.
\end{split}
\end{equation}
It is worth emphasizing the lack of variational structure of \eqref{plapnorm}, a uniformly elliptic operator  in trace form, in contrast to the regular $p$-Laplacian \eqref{plap}, a quasilinear operator in divergence form.

Both normalized operators \eqref{inflapnorm} and \eqref{plapnorm} (except for $p=2$) are undefined when $\nabla u=0$, where they have a bounded discontinuity, even if $u$ is regular. This can be remediated adapting the notion of viscosity solution (see \cite[Section 9]{CIL} and also  \cite[Chapter 2]{Giga}) using the upper and lower semicontinuous envelopes (relaxations) of the operator.

Let us recall here the definition of viscosity solution to be used in the sequel. To this end, 
given a function $h$ defined in a set $L$, we need to introduce its \emph{upper semicontinuous envelope} $h^*$ and \emph{lower semicontinuous envelope} $h_*$ defined by
\[
\begin{split}
h^*(x)=\lim_{r\downarrow0}\sup\{h(y):\ y\in\overline{B}_r(x)\cap L\}\\
h_*(x)=\lim_{r\downarrow0}\inf\{h(y):\ y\in\overline{B}_r(x)\cap L\}
\end{split}
\] 
as functions defined in $\overline{L}$.

\begin{definition}\label{def.visc.elliptic}
Let  $\mathcal{F}(\xi,X)$ be defined in a dense subset of $\mathbb{R}^n\times S^n$ with values in $\mathbb{R}$ and assume $\mathcal{F}_*<\infty$ and $\mathcal{F}^*>-\infty$ in $\mathbb{R}^n\times S^n$.

\noindent1.\quad{}An upper semicontinuous function $u:\Omega\to\mathbb{R}$ is a \emph{viscosity subsolution} of
\begin{equation}\label{E:P}
\mathcal{F}\big(\nabla u,D^{2}u\big)= f(x)
\end{equation}
in $\Omega$ if for all ${x_{0}}\in\Omega$ and $\varphi\in \mathcal{C}^{2}(\Omega)$
such that $u-\varphi$ attains a local maximum at  ${x_{0}}$,  one has
\[
\mathcal{F}_*\big(\nabla \varphi({x_{0}}),D^{2}\varphi({x_{0}})\big)\leq f({x_{0}}).
\]

\noindent2.\quad{}A lower semicontinuous function $u:\Omega\to\mathbb{R}$ is a \emph{viscosity supersolution} of \eqref{E:P} in $\Omega$ if for all ${x_{0}}\in\Omega$ and $\varphi\in \mathcal{C}^{2}(\Omega)$ such that $u-\varphi$ attains a local minimum at  ${x_{0}}$,  one has
\[
\mathcal{F}^*\big(\nabla \varphi({x_{0}}),D^{2}\varphi({x_{0}})\big)\geq f({x_{0}}).
\]

\noindent3.\quad{}We say that $u$ is a \emph{viscosity solution} of
\eqref{E:P} in $\Omega$ if it is both a viscosity subsolution and supersolution.
\end{definition}

 Given a symmetric matrix $X\in S^n$, we shall denote by
$M(X)$ and $m(X)$ its greatest and smallest eigenvalues respectively, that is,
\[
M(X)=\max_{|\xi|=1}\langle X\xi,\xi\rangle,\qquad m(X)=\min_{|\xi|=1}\langle X\xi,\xi\rangle.
\]

In the case of operators \eqref{inflapnorm} and \eqref{plapnorm}, which are continuous in $\mathbb{R}^n\setminus\{0\}\times S^n$, Definition \ref{def.visc.elliptic} can be particularized as follows.

\begin{definition}\label{def.visc.plapnorm}
Let $\Omega$ be a bounded domain and $1<p<\infty$. An upper semicontinuous function $u:\Omega\to\mathbb{R}$ is a \emph{viscosity subsolution} of
\begin{equation}\label{eq.def.plapnorm}
-\Delta_{p}^N u(x)=f(x)\qquad\text{in}\ \Omega,
\end{equation}
if for all ${x_{0}}\in\Omega$ and $\varphi\in \mathcal{C}^{2}(\Omega)$ such that $u-\varphi$ attains a local maximum at  ${x_{0}}$,  one has
\[
\left\{
\begin{split}
-&\Delta_p^{N}\varphi({x_{0}})\leq f({x_{0}})\hspace{113pt}\text{if}\ \nabla\varphi({x_{0}})\neq0,\\
-&\frac{1}{p}\Delta\varphi({x_{0}})-\frac{p-2}{p}\,M(D^2\varphi({x_{0}}))\leq f({x_{0}})\quad\text{if}\ \nabla\varphi(x_0)=0\ \text{and}\ p\in[2,\infty),\\
-&\frac{1}{p}\Delta\varphi({x_{0}})-\frac{p-2}{p}\,m(D^2\varphi({x_{0}}))\leq f({x_{0}})\hspace{13pt}\text{if}\ \nabla\varphi(x_0)=0\ \text{and}\ p\in(1,2],
\end{split}
\right.
\]
with $\Delta_{p}^Nu$ given by \eqref{plapnorm}.
A lower semicontinuous function $u:\Omega\to\mathbb{R}$ is a \emph{viscosity supersolution} of \eqref{eq.def.plapnorm} in $\Omega$  if for all ${x_{0}}\in\Omega$ and $\varphi\in \mathcal{C}^{2}(\Omega)$ such that $u-\varphi$ attains a local minimum at  ${x_{0}}$,  one has
\[
\left\{
\begin{split}
-&\Delta_p^{N}\varphi({x_{0}})\geq f({x_{0}})\hspace{113pt}\text{if}\ \nabla\varphi({x_{0}})\neq0,\\
-&\frac{1}{p}\Delta\varphi({x_{0}})-\frac{p-2}{p}\,m(D^2\varphi({x_{0}}))\geq f({x_{0}})\hspace{13pt}\text{if}\ \nabla\varphi(x_0)=0\ \text{and}\ p\in[2,\infty),\\
-&\frac{1}{p}\Delta\varphi({x_{0}})-\frac{p-2}{p}\,M(D^2\varphi({x_{0}}))\geq f({x_{0}})\hspace{11pt}\text{if}\ \nabla\varphi(x_0)=0\ \text{and}\ p\in(1,2].
\end{split}
\right.
\]

We say that $u$ is a \emph{viscosity solution} of
\eqref{eq.def.plapnorm} in $\Omega$ if it is both a viscosity subsolution and supersolution.
\end{definition}

\begin{remark}
Geometrically, the condition ``$u-\varphi$ attains a local maximum (minimum) at  ${x_{0}}$" means that up to a vertical translation, the graph of $\varphi$ touches the graph of $u$ from above (below).
\end{remark}

Notice that although this definition is slightly different to the one in \cite{PSSW2}, it turns to be equivalent, see \cite{Kawohl.Manfredi.Parviainen}.

 It is worth comparing this definition with the situation in the variational case in the range $1<p<2$. In that case the singularity in \eqref{plap} when written in trace form is not bounded and it is not possible to use the theory mentioned above. Instead, one can adopt the definition proposed in a series of papers by Birindelli and Demengel, see \cite{BD3} and the references therein. An alternative but equivalent definition in the case $f=0$ can be found in \cite[Section 2.1.3]{Giga}.

Analogously, we have the definition of solution for the normalized $\infty$-Laplacian \eqref{inflapnorm}.

\begin{definition}\label{def.visc.inflapnorm}
Let $\Omega$ be a bounded domain. An upper semicontinuous function $u:\Omega\to\mathbb{R}$ is a \emph{viscosity subsolution} of
\begin{equation}\label{eq.def.inflapnorm}
-\Delta_{\infty}^N u(x)=f(x)\qquad\text{in}\ \Omega,
\end{equation}
if for all ${x_{0}}\in\Omega$ and $\varphi\in \mathcal{C}^{2}(\Omega)$ such that $u-\varphi$ attains a local maximum at  ${x_{0}}$,  one has
\begin{equation}\label{def.subsol.inflap}
\left\{
\begin{split}
-&\Delta_\infty^{N}\varphi({x_{0}})\leq f({x_{0}}),\hspace{24pt}\text{whenever}\ \nabla\varphi({x_{0}})\neq0,\\
-&M(D^2\varphi({x_{0}}))\leq f({x_{0}}),\hspace{9pt}\text{otherwise}.
\end{split}
\right.
\end{equation}
A lower semicontinuous function $u:\Omega\to\mathbb{R}$ is a \emph{viscosity supersolution} of \eqref{eq.def.inflapnorm} in $\Omega$  if for all ${x_{0}}\in\Omega$ and $\varphi\in \mathcal{C}^{2}(\Omega)$ such that $u-\varphi$ attains a local minimum at  ${x_{0}}$,  one has
\begin{equation}\label{def.supersol.inflap}
\left\{
\begin{split}
-&\Delta_\infty^{N}\varphi({x_{0}})\geq f({x_{0}}),\hspace{24pt}\text{whenever}\ \nabla\varphi({x_{0}})\neq0,\\
-&m(D^2\varphi({x_{0}}))\geq f({x_{0}}),\quad\text{otherwise}.
\end{split}
\right.
\end{equation}
We say that $u$ is a \emph{viscosity solution} of
\eqref{eq.def.inflapnorm} in $\Omega$ if it is both a viscosity subsolution and supersolution.
\end{definition}

The above definition agrees with the one in \cite{LuWangCPDE, LuWangEJDEQ} and is slightly different from the one in \cite{PSSW}. Anyway, it is easy to see that all this definitions are equivalent.

\section{The ABP estimate for the infinity Laplacian}\label{section.ABP.inflap}

There are two parts in the proof of the classical ABP estimate. The first one is entirely geometric and allows to bound $\sup_\Omega u$ in terms of the integral of $\det D^2u$ as follows,
\begin{equation}\label{parte1}
\left(\frac{\sup_{\Omega} u- \sup_{\partial\Omega} u}{{\rm diam}(\Omega)}\right)^n \leq C(n) \int_{\{u=\Gamma(u)\}}|\det D^2u(x)|\,dx
\end{equation}
where $\Gamma(u)$ is the concave envelope of $u$ (see Definition \ref{contact.elliptic} below). The argument leading to \eqref{parte1} is entirely geometric and does not involve any equation for $u$.

In the second part  we  take advantage of the fact that we are working in the  contact set (so that the Hessian has a sign) and use  the inequality between the arithmetic and geometric means to bound the  determinant in terms of  the trace. In the particular case of the Laplacian,  the following estimate is obtained,
\begin{equation}\label{parte2}
\int_{\{u=\Gamma(u)\}}|\det D^2u(x)|\,dx\leq n^{-n}\int_{\{u=\Gamma(u)\}}\left(-{\rm trace}\, D^2u(x)\right)^n\,dx.
\end{equation}
Estimates \eqref{parte1} and \eqref{parte2} altogether yield the ABP estimate
\[
\left(\frac{\sup_{\Omega} u- \sup_{\partial\Omega} u}{{\rm diam}(\Omega)}\right)^n \leq C(n)\int_{\{u=\Gamma(u)\}}\left(-\Delta u(x)\right)^n\,dx.
\]
At this stage one can use the equation for $u$, namely $-\Delta u\leq f$ to control the right-hand side of the  estimate.

It is  apparent from the above discussion that if we are able to estimate $|\det D^2u|$ in terms of  the infinity Laplacian we shall get an ABP-type estimate. However, a different kind of estimate is expected to hold since, as we have mentioned in the introduction, the classical ABP estimate is false for the infinity Laplacian (see Section \ref{section.examples}).

\medskip

We have to recall some necessary notions (see \cite{CC}).

\begin{definition}\label{contact.elliptic}

Let $u$ be a continuous function in an open convex set $A$. The concave envelope of $u$ in $A$ is defined by
\[
\begin{split}
\Gamma(u)(x)&=\inf_{w}\big\{w(x):\ w\geq u \ \text{in}\ A,\ w\ \text{concave in}\ A\big\}\\
&=\inf_{L}\big\{L(x): \ L\geq u\ \text{in} \ A,\ L\ \text{is affine}\big\}
\end{split}
\]
for $x\in A$.
The upper contact set of $u$ is defined as,
\[
C^+(u)= \{x\in A :\
u(x) = \Gamma(u)(x)\}.
\]
\end{definition}




The following is one of the  main results in this section.
%
%
%

\begin{theorem}\label{ABP.inflapnorm.rigorous}
Let $f\in \mathcal{C}(\Omega)$ and consider 
 $u\in \mathcal{C}(\overline \Omega)$  
that satisfies
\begin{equation}\label{eq.th.ABP.sub}
-\Delta_\infty^N u \leq f(x)\quad\text{in}\ \Omega
\end{equation}
in the viscosity sense. Then, we have
\begin{equation}\label{ABP.est.inflapnorm.sub}
\big(\sup_{\Omega }u-\sup_{\partial\Omega}u^{+}\big)^2\leq 2\, d^2 \int _{\sup_{\partial\Omega}u^+}^{\sup_{\Omega}u}
\|f^+\cdot1_{C^+(u)}\|_{L^\infty(\{u^{+}=r\})}\,dr,
\end{equation}
where  $d={\rm diam}\left( \Omega \right)$. Analogously, whenever  
\begin{equation}\label{eq.th.ABP.super}
-\Delta_\infty^N u \geq f(x)\quad\text{in}\ \Omega
\end{equation}
in the viscosity sense, the following estimate holds, 
\begin{equation}\label{ABP.est.inflapnorm.sup}
\big(\sup_{\Omega }u^{-}-\sup_{\partial\Omega}u^{-}\big)^2\leq 2\, d^2 \int _{\sup_{\partial\Omega }u^{-}}^{\sup_{\Omega }u^{-}}
\|f^-\cdot1_{C^+(-u)}\|_{L^\infty(\{u^{-}=r\})}\,dr.
\end{equation}
\end{theorem}

\begin{proof}
We prove the first inequality, since the second one is similar. The proof is divided in two steps, in the  first one we prove the result in the case of a semiconvex subsolution $v$ of \eqref{eq.th.ABP.sub} and then we shall remove this assumption in the second step.

\medskip

\noindent \textsc{Step 1. } We can assume without loss of generality that $\sup_{\partial\Omega} v^+=0$   and $v(x_0)>0$ for some $x_0 \in \Omega$. Let us consider the function  $v^+$ extended by $0$ to the whole $\mathbb{R}^n$. Notice that  $v^+$ is also semiconvex 
and satisfies
\[
-\Delta_\infty^N v^+\leq f^+ \qquad\text{in}\ \mathbb{R}^n.
\]
Let $\Omega^*={\rm conv} (\Omega) $, the convex hull of $\Omega$ and for every  $\sigma>0$ let us  consider  $\Gamma_\sigma (v^+)$, the concave envelope of $v^+$ in $\Omega^*_\sigma$, where
$$
\Omega^*_\sigma=\{x \in\mathbb{R}^n \textrm{ such that } {\rm dist}(x,\Omega^*)\le\sigma\}. 
$$
Since we are assuming that $v$ is semiconvex, Lemma \ref{lemma.cici} implies that $\Gamma_\sigma (v^+)$ is $\mathcal{C}_{\rm loc}^{1,1}(\Omega^*_\sigma)$.

Our aim is to estimate 
\begin{equation}\label{ohyeah}
\int_{\nabla \Gamma_\sigma (v^+)(\Omega^*_\sigma)} |\xi|^{2-n}d\xi,
\end{equation}
both from above and from below. For the sake of brevity, let us denote in the sequel 
\[
M:=\sup_{\Omega^*_\sigma}\Gamma_\sigma (v^+)=\sup_\Omega v^+=\sup_\Omega v
\]
and
\[
A_\sigma =C_\sigma ^+(v^+)\cap\left\{\text{Points of twice differentiability for }\Gamma_\sigma (v^+)\right\},
\]
where $C_\sigma ^+(v^+)$ is the set of points in $\Omega^*_\sigma$ where $v^+=\Gamma_\sigma(v^+)$.
First we estimate \eqref{ohyeah} from above. Applying  the Area Formula for Lipschitz maps  \cite[Theorem 3.2.5]{Federer}, Lemma \ref{lemma.cici}, and then the Coarea Formula \cite[Theorem 3.2.11]{Federer}, we get,
\begin{equation}\label{bungabungalele}
\begin{split}
& \int_{\nabla \Gamma_\sigma (v^+)(\Omega^*_\sigma )} |\xi|^{2-n}d\xi\leq \int_{\Omega^*_\sigma} |\nabla \Gamma_\sigma(v^+)|^{2-n}\cdot\det(-D^2 \Gamma_\sigma (v^+))\,dx\\
&= \int_{\Omega^*_\sigma}1_{A_\sigma}(x)\cdot |\nabla\Gamma_\sigma(v^+)|^{2-n}\cdot \det(-D^2 \Gamma_\sigma(v^+))\,dx\\
&=\int_{0}^{M}\int_{\{\Gamma_\sigma(v^+)=r\}}1_{A_\sigma} (x)\cdot|\nabla\Gamma_\sigma (v^+)|^{1-n} \cdot\det(-D^2 \Gamma_\sigma(v^+))\,d\mathcal{H}^{n-1}dr.
\end{split}
\end{equation}

From  Hadamard's formula \cite[Section 2.10]{BB} for the determinant of a positive definite matrix and Proposition \ref{proposition.hardcore}, we get that for a.e. $r\in(0,M)$, 
\begin{equation}\label{Hadamard.Guido}
\det(- D^2 \Gamma_\sigma (v^+))\leq -\Delta_\infty^N  \Gamma_\sigma (v^+)\cdot |\nabla_\sigma \Gamma(v^+)|^{n-1}\cdot \prod_{i=1}^{n-1} \kappa_i
\end{equation}
holds $\mathcal{H}^{n-1}\text{-a.e.}$ in $\{\Gamma_\sigma(v^+)=r\}$, for $\kappa_i(x)$ the principal curvatures of $\{\Gamma_\sigma(v^+)=r\}$ at $x$, as defined in Proposition \ref{proposition.hardcore}.
Hence, from \eqref{bungabungalele}, \eqref{Hadamard.Guido} and Lemma \ref{pointwise} we have,
\begin{equation}\label{35centcoffee}
\begin{split}
& \int_{\nabla \Gamma_\sigma (v^+)(\Omega^*_\sigma)} |\xi|^{2-n}d\xi\\
 &\hspace{20pt}\leq  \int_{0}^{M} \int_{\{\Gamma_\sigma (v^+)=r\}} 1_{A_\sigma} (x)\cdot f^+(x)\cdot\prod_{i=1}^{n-1}\kappa_i(x)\;d\mathcal{H}^{n-1}dr\\
 &\hspace{20pt}\leq \int_{0}^{M}\|f^+\,\cdot1_{C_\sigma ^+(v^+)}\|_{L^\infty(\{v^+=r\})}\int_{\{\Gamma_\sigma (v^+)=r\}}\prod_{i=1}^{n-1}\kappa_i(x)\;d\mathcal{H}^{n-1}\,dr\\
 &\hspace{20pt}\leq\mathcal{H}^{n-1}(\partial B_1(0)) \,\int_{0}^{M}\|f^+\,\cdot1_{C^+_\sigma(v^+)}\|_{L^\infty(\{v^+=r\})}\,dr.
\end{split}
\end{equation}
The last step follows from the Gauss-Bonnet theorem, Theorem \ref{GBthheorem}.

\medskip

To estimate  \eqref{ohyeah} from below, we observe that 
\[
\partial K_{x_0,\sigma}(\Omega_\sigma^*)\subset\nabla \Gamma_\sigma (v^+)(\Omega_\sigma^*) 
\]
where $x_0 $ is a maximum point in $\Omega^*$ for $v^+ (x)$, $K_{x_0,\sigma}(x)$ is the  concave cone with vertex in $(x_0, v^+(x_0))$  and base $\Omega_\sigma^*$ and $\partial K_{x_0,\sigma}$ its super-differential (see \cite[Section 1.4]{Gutierrez}). Collecting the  previous information we get
\begin{multline}\label{marta}
\int_{\partial K_{x_0,\sigma}(\Omega_\sigma^*)}|\xi|^{2-n}d\xi\\
 \le \mathcal{H}^{n-1}(\partial B_1(0)) \,\int_{0}^{M}\|f^+\,\cdot1_{C^+_\sigma(v^+)}\|_{L^\infty(\{v^+=r\})}\,dr.
\end{multline}  
 We want now to pass to the limit as $\sigma $ goes to $0$. It is easy to see that $\partial K_{x_0,\sigma}(\Omega_\sigma^*) \uparrow \partial K_{x_0}(\Omega^*)$. We need now to estimate the limit of the right-hand side, to do this notice that $\Gamma_\sigma (v^+)$ is a  decreasing sequence of concave functions which thus converges to a concave function $w$ defined on $\Omega$ such that $w\ge \Gamma(v^+)\ge v^+$, so in particular
 \[
 \limsup_{\sigma \to 0} C^+_\sigma (v^+) \subset \{w=v^+\}\subset C^+(v^+).
\]
The last equation  implies that for every $r\in (0,M)$
\[
 \limsup_{\sigma \to 0} \|f^+ \cdot1_{C^+_\sigma(v^+)}\|_{L^\infty(\{v^+=r\})}\leq \|f^+\cdot 1_{C^+(v^+)}\|_{L^\infty(\{v^+=r\})}.
\]
Thus, \eqref{marta} becomes 
\begin{multline}\label{marta2}
\int_{\partial K_{x_0}(\Omega^*)}|\xi|^{2-n}d\xi \\
\le \mathcal{H}^{n-1}(\partial B_1(0)) \,\int_{0}^{M}\|f^+\,\cdot1_{C^+(v^+)}\|_{L^\infty(\{v^+=r\})}\,dr.
\end{multline}

It is now very easy to see that 
\[
B_{\frac{M}{d}}(0)\subset \partial K_{x_0}(\Omega^*).
\]
where $d={\rm diam}(\Omega^*)={\rm diam}(\Omega)$ and $M=\sup_{\Omega} v=\sup_{\Omega^*} K_{x_0} (x)$. Consequently,
\begin{equation}\label{estimate.below}
\begin{split}
 \int_{\partial K_{x_0}(\Omega^*)} |\xi|^{2-n}d\xi&\geq \int_{B_{\frac{M}{d}}(0)} |\xi|^{2-n}\,d\xi=\frac {\mathcal{H}^{n-1}(\partial B_1(0))}{2}\left(\frac{M}{d}\right)^2.
\end{split}
\end{equation}
Collecting \eqref{marta2} and \eqref{estimate.below} we arrive at
$$
(\sup_{\Omega} v^+)^2\leq 2 d^2\int_0^{\sup_{\Omega}v^+}\|f^+\,\cdot1_{C^+(v^+)}\|_{L^\infty(\{v^+=r\})}\, dr.
$$
\bigskip

\noindent\textsc{Step 2. } Now we consider the case of a general viscosity subsolution  $u$ of \eqref{eq.th.ABP.sub}. Then, the sup-convolution of $u$, defined as
\begin{equation}\label{u.eps}
u^\varepsilon (x)=\sup_{y\in  \overline\Omega}\left\{u(y)-\frac{1}{2\varepsilon}|x-y|^2\right\},
\end{equation}
 is semiconvex and verifies
$$
-\Delta_\infty^N {u^\varepsilon}  \leq f_\varepsilon(x):= \sup\big\{f(y):\ y\in B_{2(\varepsilon\|u\|_{L^\infty(\Omega)})^{1/2}}(x)\big\}\qquad \text{in}\,\,\, \Omega_\varepsilon
$$
where 
$$
\Omega_\varepsilon:=\{ x\in \Omega : {\rm dist}(x,\partial \Omega)>2(\varepsilon\|u\|_{L^\infty(\Omega)})^{1/2}\}.
$$
Moreover, $u\leq u^\varepsilon$ and $u^\varepsilon\to u$ uniformly as $\varepsilon\to0$ in $\overline\Omega$, see \cite[Chapter 5]{CC}.

From the previous step we know that
\begin{equation*}
\begin{split}
&\big(\sup_{\Omega_\varepsilon} ({u^\varepsilon})^+-\sup_{\partial \Omega_\varepsilon} ({u^\varepsilon})^+\big)^2\\
&\hspace{35pt}\leq 2\,{\rm diam}(\Omega_\varepsilon)^2\, \int_{\sup_{\partial \Omega_\varepsilon}     (u^\varepsilon )^+}^{\sup_{\Omega_\varepsilon} ({u^\varepsilon})^+}\|f^+_\varepsilon\,\cdot1_{C^+((u^\varepsilon)^+) }\|_{L^\infty(\{({u^\varepsilon})^+=r\})}dr.
\end{split}
\end{equation*} 
If we are able to pass to the limit in the previous inequality we shall obtain the desired estimate.  It is easy to see that actually we can pass to the limit  once we have proved that
\begin{equation*}
\limsup_{\varepsilon\rightarrow 0}\|f_\varepsilon^+ \cdot1_{C^+ ((u^\varepsilon)^+)}\|_{L^\infty(\{(u^\varepsilon)^+=r\})}\leq \|f^+\cdot
 1_{C^+ (u^+)}\|_{L^\infty(\{u^+=r\})}.
\end{equation*}
Let us  consider  a sequence $\{\varepsilon_j\}$ such that the $\limsup$ is realized and let $x_{\varepsilon_j}\in\{({u^{\varepsilon_j}})^+=r\}$ such that
\[
\|f_{\varepsilon_j}^+\,\cdot1_{C^+ ((u^{\varepsilon_j})^+)}\|_{L^\infty(\{(u^{\varepsilon_j})^+=r\})}=f
^+_{\varepsilon_j} (x_{\varepsilon_j})\,\cdot1_{C^+((u^{\varepsilon_j})^+)}(x_{\varepsilon_j}).
\]
Up to a subsequence, still denoted with $\varepsilon_j$, we can assume that $x_{\varepsilon_j}$ converges to $x_0\in \overline{\Omega}$. Since $x_{\varepsilon_j}$ belongs to $\{(u^{\varepsilon_j})^+=r\}$ and $u^{\varepsilon_j}$ converge uniformly to $u$ we obtain $u^+(x_0)=r$, in particular $x_0 \notin\partial \Omega$ for any $r > \sup_{\partial \Omega} u$.


Now since $f_{\varepsilon_j}$ converge uniformly to $f$ on compact subsets of $\Omega$ and 
$$
\limsup_{\varepsilon_j\rightarrow 0} C^+((u^{\varepsilon_j})^+)\subset C^+(u^+) 
$$
thanks to \cite[Lemma A.1]{CaCrKoSw}, we get
\begin{multline}
\limsup_{\varepsilon_j\rightarrow 0}f_{\varepsilon_j}^+(x_{\varepsilon_j})\,\cdot1_{C^+ ((u^{\varepsilon_j})^+)}(x_{\varepsilon_j})\\
\leq f^+(x_0)\,\cdot1_{C^+(u^+)}(x_0)
\leq \|f^+\,\cdot1_{C^+(u^+)}\|_{L^{\infty}(\{u^+=r\})}
\end{multline}
and hence the thesis follows. 
\end{proof}

The following is an easy consequence of Theorem \ref{ABP.inflapnorm.rigorous}.
\begin{corollary}
Let $f\in \mathcal{C}(\Omega)$ and consider 
 $u\in \mathcal{C}(\overline \Omega)$  
that satisfies
\begin{equation*}
-\Delta_\infty^N u \leq f(x)\quad\text{in}\ \Omega
\end{equation*}
in the viscosity sense. Then, the following estimate holds, 
\[
\sup_{\Omega }u\leq\sup_{\partial\Omega}u^{+}+
 2\, d^2\, 
\|f^+\|_{L^\infty(C^+(u))}.
\]
Analogously, whenever  
\begin{equation*}
-\Delta_\infty^N u \geq f(x)\quad\text{in}\ \Omega
\end{equation*}
in the viscosity sense,
\[
\sup_{\Omega }u^{-}\leq\sup_{\partial\Omega}u^{-}+
 2\, d^2\, 
\|f^-\|_{L^\infty(C^+(-u))},
\]
where  $d={\rm diam}\left( \Omega \right)$. 
\end{corollary}

\begin{remark}
Similar results can be derived from Theorems \ref{thm.ABP.inflap3homog}, \ref{ABP.plapnorm.stable}, and \ref{thm.ABP.full.plap} below.
\end{remark}

If  we replace \eqref{ohyeah} by
\[
\int_{\nabla \Gamma_\sigma (v^+)(\Omega^*_\sigma)} |\xi|^{4-n}d\xi,
\]
the  same proof yields the following result.

\begin{theorem}\label{thm.ABP.inflap3homog}
Let $f\in \mathcal{C}(\Omega)$ and consider 
 $u\in \mathcal{C}(\overline \Omega)$  
that satisfies
\[
-\Delta_\infty u \leq f(x)\quad\text{in}\ \Omega
\]
in the viscosity sense. Then, we have
\[
\big(\sup_{\Omega }u-\sup_{\partial\Omega}u^{+}\big)^4\leq 4\cdot d^4\int _{\sup_{\partial\Omega}u^+}^{\sup_{\Omega}u}
\|f^+\cdot1_{C^+(u)}\|_{L^\infty(\{u^{+}=r\})}\,dr,
\]
where  $d={\rm diam}\left( \Omega \right)$. Analogously, whenever  
\[
-\Delta_\infty u \geq f(x)\quad\text{in}\ \Omega
\]
in the viscosity sense, the following estimate holds, 
\[
\big(\sup_{\Omega }u^{-}-\sup_{\partial\Omega}u^{-}\big)^4\leq 4\cdot d^4\int _{\sup_{\partial\Omega }u^{-}}^{\sup_{\Omega }u^{-}}
\|f^-\cdot1_{C^+(-u)}\|_{L^\infty(\{u^{-}=r\})}\,dr.
\]
\end{theorem}

\medskip

\section{An stable  ABP for the normalized p-Laplacian}\label{section.ABP.plap}

The following is one of the  main results in this section.

\begin{theorem}\label{ABP.plapnorm.stable}
Let $1<p<\infty $, $f\in \mathcal{C}(\Omega)$ and consider 
 $u\in \mathcal{C}(\overline \Omega)$  
that satisfies
\begin{equation*}
-\Delta_p^N u \leq f(x)\quad\text{in}\ \Omega
\end{equation*}
in the viscosity sense. Then, we have
\begin{equation}\label{ABP.est.plapnorm.sub}
\big(\sup_{\Omega }u-\sup_{\partial\Omega}u^{+}\big)^2\leq\frac{2\,p\,d^2}{p-1}\, \int _{\sup_{\partial\Omega}u^+}^{\sup_{\Omega}u}
\|f^+\cdot1_{C^+(u)}\|_{L^\infty(\{u^{+}=r\})}\,dr,
\end{equation}
where  $d={\rm diam}\left( \Omega \right)$. Analogously, whenever  
\begin{equation*}
-\Delta_p^N u \geq f(x)\quad\text{in}\ \Omega
\end{equation*}
in the viscosity sense, the following estimate holds, 
\begin{equation}\label{ABP.est.plapnorm.sup}
\big(\sup_{\Omega }u^{-}-\sup_{\partial\Omega}u^{-}\big)^2\leq \frac{2\,p\,d^2}{p-1}\, \int _{\sup_{\partial\Omega }u^{-}}^{\sup_{\Omega }u^{-}}
\|f^-\cdot1_{C^+(-u)}\|_{L^\infty(\{u^{-}=r\})}\,dr.
\end{equation}
\end{theorem}

\begin{proof}
We only sketch the proof of the first inequality, since the second one is similar. The
argument goes exactly in the same way of the one of Theorem \ref{ABP.inflapnorm.rigorous} with the only observation that thanks to the concavity of $\Gamma_\sigma(u^+) $ and Lemmas \ref{lemma.cici} and \ref{pointwise} and Proposition \ref{proposition.hardcore}, at almost every contact point $x \in C_\sigma^+(u^+)$ one has, 
\[
\begin{split}
-\Big(\frac{p-1}{p}\Big)&\,\Delta_\infty^N \Gamma_\sigma(u^+)(x)\\
&\leq-\Big(\frac{p-1}{p}\Big)\,\Delta_\infty^N\Gamma_\sigma(u^+)(x)+\frac 1 p\,|\nabla \Gamma_\sigma(u^+)(x)|\,\sum_{i=1}^{n-1}\kappa_i(x)\\
&=-\Delta_{p}^N \Gamma_\sigma(u^+)(x)\leq f^+(x),
\end{split}
\]
where $\kappa_i(x)$ $i=1,\ldots,n-1$ denote, as before, the principal curvatures at $x$ of the level set of $\Gamma_\sigma (u^+)$ passing through $x$.
\end{proof}

\begin{remark}Estimates \eqref{ABP.est.plapnorm.sub} and \eqref{ABP.est.plapnorm.sup} are stable as $p\to\infty$ and, moreover, allow to recover  estimates \eqref{ABP.est.inflapnorm.sub} and \eqref{ABP.est.inflapnorm.sup} in the limit.
\end{remark}

It is interesting to compare estimates \eqref{ABP.est.plapnorm.sub} and \eqref{ABP.est.plapnorm.sup} with the usual ABP estimates that are not stable in the limit.

The following theorem can be easily proved with the techniques of \cite{ACP,CC, DFQ,Imbert}, the proof we sketch here actually allows a comparison with the previous one, see the remarks at the end of the proof.

\begin{theorem}\label{ABP.plapnorm.norm.n}
Let $1<p<\infty$, $f\in\mathcal{C}(\Omega)$ and consider $u\in\mathcal{C}(\overline\Omega)$  that satisfies
\[
-\Delta_p^N u \leq f(x)\quad\text{in}\ \Omega,
\]
in the viscosity sense. Then, the following estimate holds, 
\[
\sup_{\Omega }u\leq \sup_{\partial\Omega}u^{+}+ \frac{d\,p}{n|B_1(0)|^\frac{1}{n}(p-1)^\frac{1}{n}} \|f^+\|_{L^n(C^+(u))},
\]
where  $d={\rm diam}\left( \Omega \right)$.

Analogously, whenever  $u\in\mathcal{C}(\overline\Omega)$  satisfies
\[
-\Delta_p^N u \geq f(x)\quad\text{in}\ \Omega
\]
in the viscosity sense, the following estimate holds, 
\[
\sup_{\Omega }u^{-}\leq \sup_{\partial\Omega}u^{-}+ \frac{d\,p}{n|B_1(0)|^\frac{1}{n}(p-1)^\frac{1}{n}} \|f^-\|_{L^n(C^+(-u))},
\]
\end{theorem}

\begin{proof}
Again we don't take care of regularity issues which can be easily handled with the technique we used in the proof of Theorem \ref{ABP.inflapnorm.rigorous}.

The standard ABP argument yields,
\[
\begin{split}
&\left(\frac{\sup_\Omega u-\sup_{\partial \Omega}u^+}{d}\right)^{n}|B_1(0)|\\
&\hspace{45pt}\leq|\nabla \Gamma (u^+) \left( C^{+}\left( u \right) \right)|\leq\int_{C^{+}\left( u^+ \right)}\det (-D^2u)\,dx.
\end{split}
\]
Proposition \ref{proposition.hardcore} yields
\[
\int_{C^{+}\left( u^+ \right)}\det (-D^2u)\,dx \leq \int_{C^{+}\left( u^+ \right)}-\Delta_{\infty}^N\Gamma (u^+)\cdot|\nabla \Gamma (u^+)|^{n-1}\cdot\prod_{i=1}^{n-1} \kappa_i\,dx.
\]
Now, we multiply and divide by $p-1$ the right-hand side of the previous inequality, and then apply the inequality between the aritmethic and geometric mean inequalities. We get,
\[
\begin{split}
\int_{C^{+}\left( u^+ \right)}&-\Delta_{\infty}^N\Gamma (u^+)\cdot|\nabla \Gamma (u^+)|^{n-1}\cdot\prod_{i=1}^{n-1} \kappa_i\,dx\\
&\leq
\frac{1}{(p-1)n^n}\int_{C^{+}\left( u^+ \right)}
\Big(-(p-1)\,\Delta_{\infty}^N\Gamma( u^+)+|\nabla \Gamma (u^+)|\,\sum_{i=1}^{n-1}\kappa_i\Big)^n\,dx\\
&=
\frac{p^n}{(p-1)n^n}\int_{C^{+}\left( u^+ \right)}
\big(-\Delta_p^N\Gamma (u^+) \big)^n\,dx\\
&\leq
\frac{p^n}{(p-1)n^n}\int_{C^{+}\left( u \right)}
\big(f^+\big)^n\,dx.
\end{split}
\]
\end{proof}
Some comments are in order. First, it is important to notice that the estimate in Theorem \ref{ABP.plapnorm.norm.n} is not stable in $p$, as  the resulting constant blows up as $p\to\infty$. Hence we don't recover any estimate in the limit, in contrast to Theorem \ref{ABP.plapnorm.stable} that yields Theorem \ref{ABP.inflapnorm.rigorous} as a limit case.

This fact can be understood by comparing the  proofs of Theorems \ref{ABP.plapnorm.stable} and \ref{ABP.plapnorm.norm.n}. In fact, they make apparent that the infinity Laplacian controls one single direction, the direction of steepest descent, totally neglecting the curvature of the level sets (whose average is in turn controlled by the Gauss-Bonet Theorem). The $p$-Laplacian, instead, is a weighted mean of the curvatures of the level sets and the normal direction. In particular, when $p=2$, the classical Laplacian, all the directions are weighted the same way.

In this sense, it is completely natural that in the $p$-Laplacian case the ABP involves an integral norm, an average in all directions, while in the infinity Laplacian case the average is somehow unidimensional.

\medskip

If we replace 
\eqref{ohyeah} by
\[
\int_{\nabla \Gamma_\sigma (v^+)(\Omega^*_\sigma)} |\xi|^{p-n}d\xi,
\]
the  same proof of Theorem \ref{ABP.plapnorm.stable} yields similar estimates for the variational $p$-Laplacian \eqref{plap}.

\begin{theorem}\label{thm.ABP.full.plap}

Let $1<p<\infty $, $f\in \mathcal{C}(\Omega)$ and consider 
 $u\in \mathcal{C}(\overline \Omega)$  
that satisfies
\begin{equation*}
-\Delta_p u \leq f(x)\quad\text{in}\ \Omega
\end{equation*}
in the viscosity sense. Then, we have
\begin{equation*}
\big(\sup_{\Omega }u-\sup_{\partial\Omega}u^{+}\big)^p\leq\frac{ p \,d^p}{p-1}\, \int _{\sup_{\partial\Omega}u^+}^{\sup_{\Omega}u}
\|f^+\cdot1_{C^+(u)}\|_{L^\infty(\{u^{+}=r\})}\,dr,
\end{equation*}
where  $d={\rm diam}\left( \Omega \right)$. Analogously, whenever  
\begin{equation*}
-\Delta_p u \geq f(x)\quad\text{in}\ \Omega
\end{equation*}
in the viscosity sense, the following estimate holds, 
\begin{equation*}
\big(\sup_{\Omega }u^{-}-\sup_{\partial\Omega}u^{-}\big)^p\leq \frac{ p \,d^p}{p-1}\, \int _{\sup_{\partial\Omega }u^{-}}^{\sup_{\Omega }u^{-}}
\|f^-\cdot1_{C^+(-u)}\|_{L^\infty(\{u^{-}=r\})}\,dr.
\end{equation*}

\end{theorem}

\medskip


\section{Stable estimates of the modulus of continuity for $p\leq\infty$}\label{section.modulus.cont}

In this section we obtain H\"older estimates for solutions of the normalized $p$-Laplacian as well as Lipschitz estimates for the normalized $\infty$-Laplacian. The main interest of this estimates is that they are stable in $p$, and apply to the whole range $n<p\leq\infty$ with all the parameters involved varying continuously. 

As mentioned in the introduction, the well-known estimates in \cite{CC} apply whenever $p<\infty$, but degenerate as $p\to\infty$ as they depend upon the ratio between the ellipticity constants (see also \cite{GT}, Section 9.7 and 9.8), which in this case is $p-1$ and  blows-up as $p\to\infty$.


For the sake of a unified presentation, throughout this section we are going to denote
\begin{equation}\label{alphap}
\alpha_p=
\left\{\begin{split}
&\frac{n-1}{p-1}\qquad\text{for}\ n<p<\infty\\
&0\hspace{44.5pt}\text{for}\ p=\infty
\end{split}\right.
\end{equation}
and
\begin{equation}\label{C_p}
C_p=
\left\{\begin{split}
&\frac{p}{p+n-2}\qquad\text{for}\ n<p<\infty\\
&1\hspace{63pt}\text{for}\ p=\infty.
\end{split}\right.
\end{equation}

\begin{theorem}\label{uniform}
Let $\Omega$ be a bounded domain,  $2\leq n<p\leq\infty$, and  $u $ a viscosity solution of
\begin{equation}\label{Dir}
-\Delta_p^{N} u=f \quad\textrm{in $\Omega$}.
\end{equation}
Then, for any $x\in\Omega$
\[
\begin{split}
\frac{|u(y)-u(x)|}{|y-x|^{1-\alpha_{p}}}\leq 
\sup_{z\in\partial\Omega}
\frac{|u(z)-u(x)|}{|z-x|^{1-\alpha_{p}}}
+
\frac{C_p}{1-\alpha_p}\;{\rm diam}(\Omega)^{1+\alpha_p}\,\|f\|_{L^\infty(\Omega)} 
\end{split}
\]
for every $y \in \overline \Omega$, with $\alpha_p$ and $C_p$ defined in \eqref{alphap} and \eqref{C_p} respectively.
\end{theorem}

\begin{remark}
Compare with \cite[Lemma 2.5]{Crandall.Evans.Gariepy}.
\end{remark}

The following estimate is an easy consequence of Theorem \ref{uniform}.

\begin{corollary}
Let $\Omega$ be a bounded domain,  $2\leq n<p\leq\infty$, and  $u $ a viscosity solution of \eqref{Dir}.
Let $\alpha_p$ and $C_p$ defined by \eqref{alphap} and \eqref{C_p}. Then, for any $x\in\Omega$ we have that
\[
\frac{|u(y)-u(x)|}{|y-x|^{1-\alpha_p}}
\le \frac{2\|u\|_{L^\infty(\Omega)}}{{\rm dist}(x,\partial\Omega)^{1-\alpha_p}}
+\frac{C_p}{1-\alpha_p}\;{\rm diam}(\Omega)^{1+\alpha_p} \,\|f\|_{L^\infty(\Omega)} 
\]
for every $y \in \overline \Omega$.
\end{corollary}

We also have global estimates for the Dirichlet problem.
\begin{theorem}\label{uniform2}
Let $\Omega$ be a bounded domain,  $2\leq n<p\leq\infty$, and  $u $ a viscosity solution of
\begin{equation}\label{Dirichlet}
\begin{cases}
-\Delta_p^{N} u=f &\textrm{in $\Omega$}\\
u=g &\textrm{on $\partial \Omega$}
\end{cases}
\end{equation}
with $g\in\mathcal{C}^{0,1-\alpha_p}(\partial\Omega)$ and $\alpha_p$ as in \eqref{alphap}. Then, for every $x,y \in \overline \Omega$,
\begin{equation}\label{modulus.continuity}
\frac{|u(x)-u(y)|}{|x-y|^{1-\alpha_p}}
\le
C_p \,
 \|f\|_{L^\infty(\Omega)}\,{\rm diam}(\Omega)^{1+\alpha_p}+[g]_{1-\alpha_p,\partial\Omega}
\end{equation}
with $C_p$ as in \eqref{C_p}.
\end{theorem}

\begin{corollary}\label{C.alpha.estimates}
Let $\Omega$ be a bounded domain,  $2\leq n<p\leq\infty$, and  $u $ a viscosity solution of \eqref{Dirichlet} with $f\in\mathcal{C}(\Omega)$, $g\in\mathcal{C}^{0,1-\alpha_p}(\partial\Omega)$ and $\alpha_p,\ C_p$ as in  \eqref{alphap} and \eqref{C_p} respectively. Then,
\[
\begin{split}
&\|u\|_{\mathcal{C}^{0,1-\alpha_p}(\Omega)}=
\|u\|_{L^\infty(\Omega)}+\sup_{\overset{x,y\in\Omega}{
x\neq y}}\frac{|u(x)-u(y)|}{|x-y|^{1-\alpha_p}}\\
&\quad\leq\|g\|_{\mathcal{C}^{0,1-\alpha_p}(\partial\Omega)}+\left( 2\,\tilde C_p\,{\rm diam}(\Omega)^2+C_p \,{\rm diam}(\Omega)^{1+\alpha_p}\right)\,
 \|f\|_{L^\infty(\Omega)}
 \end{split}
\]
with
\begin{equation}\label{C.p.tilde}
\tilde C_p=
\left\{\begin{split}
&\frac{p}{p-1}\qquad\text{for}\ n<p<\infty\\
&1\hspace{43pt}\text{for}\ p=\infty.
\end{split}\right.
\end{equation}
\end{corollary}

\begin{proof}
From the ABP estimates \eqref{ABP.est.inflapnorm.sub} and \eqref{ABP.est.plapnorm.sub} in Theorems \ref{ABP.inflapnorm.rigorous} and \ref{ABP.plapnorm.stable}, we have
\[
\|u\|_{L^\infty(\Omega) }\leq\|g\|_{L^\infty(\partial\Omega)}+ 2\,\tilde C_p\,{\rm diam}(\Omega)^2\,
\|f\|_{L^\infty(C^+(u))}.
\]
for $\tilde C_p$ as in \eqref{C.p.tilde}.
This estimate together with \eqref{modulus.continuity} yields the result.
\end{proof}

\bigskip

We need some lemmas in the proof of Theorems \ref{uniform} and \ref{uniform2}.

\begin{lemma}\label{fs} Consider the function 
\[
v(x)=h+\frac{A}{1-\alpha_{p}}|x-y|^{1-\alpha_{p}}-\frac{B}{2}|x-y|^2
\]
with $A,B>0$, $h\in\mathbb{R}$ and $y \in \mathbb R^n$. Let $C_p$ as defined in \eqref{C_p}. Then 
\[
-\Delta_p^{N} v(x)= C_p^{-1} B
\]
for $x\ne y$ such that $\nabla v(x)\ne 0$.
\end{lemma}

\begin{proof}
 For $x\ne y$ we have
 \[
 \nabla v(x)=\big(A\,|x-y|^{-1-\alpha_p}-B\big)\,(x-y)
 \]
 and
 \[
D^2 v(x)=\big(A\,|x-y|^{-1-\alpha_p}-B\big)\,Id-A\,(1+\alpha_p)\,|x-y|^{-1-\alpha_p}\frac{(x-y)\otimes(x-y)}{|x-y|^2}.
 \]
 Then,
\[
\Delta_\infty^{N} v(x)= -A\,\alpha_p\,|x-y|^{-1-\alpha_p}-B
\] 
and the proof is complete in the case $p=\infty$. If $p<\infty$,
\[
\begin{split}
\Delta v(x)={\rm trace}(D^2 v(x))=A\,(n-1-\alpha_p)\,|x-y|^{-1-\alpha_p}-nB
\end{split}
\]
and hence
\[
\begin{split}
\Delta_p^{N} v(x)&=\frac 1 p \Delta v(x)+\frac{p-2}{p}\Delta_\infty^{N} v(x)=-C_p^{-1} B.\qedhere
\end{split}
\]
\end{proof}

\begin{remark}
Notice that when $p=\infty$ this functions $v$ are polar quadratic polynomials as defined in \cite{armstrong.smart,LuWangCPDE}, that replace cones as the functions to which compare in the context of the non-homogeneous normalized infinity Laplacian.
\end{remark}

\begin{remark}
Both operators $\Delta_p^N$ and $\Delta_\infty^N$ are linear when applied to radial functions. This fact  motivates using cusps and cones (as they are $p$-Harmonic away from the vertex) in  Lemma \ref{fs} altogether with a quadratic perturbation.  
\end{remark}

\begin{lemma}\label{gradient.non.zero}
Let $\Omega$ be a bounded domain, $x_0\in\Omega$, $p> n$, 
\[
A>
B\,{\rm diam}(\Omega)^{1+\alpha_{p}},
\] 
and $B>0$. Then, the function
\[
v(x)=h+\frac{A}{1-\alpha_{p}}|x-x_0|^{1-\alpha_{p}}-\frac{B}{2}|x-x_0|^2
\]
satisfies $\nabla v(x)\neq0$ for any $x\in \overline\Omega\setminus\{x_0\}$.
\end{lemma}

\begin{proof}
A direct computation shows that for any $x\in \overline\Omega\setminus\{x_0\}$ we have that
\[
\begin{split}
|\nabla v(x)|&=|x-x_0|\cdot\big|A\,|x-x_0|^{-1-\alpha_{p}}-B\big|\\
&\ge |x-x_0|\cdot\big(A\,{\rm diam}(\Omega)^{-1-\alpha_{p}}-B \big)>0
\end{split}
\]
by our hypothesis on the size of $A$. 
\end{proof}

\begin{lemma}\label{fundam}
Let $\Omega$ be a bounded domain, $x_0\in\Omega$, $p> n\geq2$, 
\begin{equation}\label{defnA}
A\geq(1-\alpha_p)\sup_{z\in\partial\Omega}
\frac{|u(z)-u(x_0)|}{|z-x_0|^{1-\alpha_{p}}}+
B\;{\rm diam}(\Omega)^{1+\alpha_{p}},
\end{equation}
and $B=(1+\varepsilon)\,C_p \|f\|_{L^\infty(\Omega)}$ with $\varepsilon>0$. Suppose that $h$ is such that
\[
v(x)=h+\frac{A}{1-\alpha_{p}}|x-x_0|^{1-\alpha_{p}}-\frac{B}{2}|x-x_0|^2
\]
touches $u$ from above at some point $\tilde x\in\overline\Omega$. Then, necessarily $\tilde x\equiv x_0$ and $h=u(x_0)$. 

We arrive at the same conclusion if
 we suppose instead that $h$ is such that
\[
v(x)=h-\frac{A}{1-\alpha_{p}}|x-x_0|^{1-\alpha_{p}}+\frac{B}{2}|x-x_0|^2
\]
touches $u$ from below at some point $\tilde x\in\overline\Omega$. 
\end{lemma}

\begin{proof}
We provide the proof in the first case as the second one follows in the same way.

\noindent1.\quad{}First we are going to prove that the contact point $\tilde x\notin\partial\Omega$. Assume for contradiction that $\tilde x\in\partial\Omega$. From the contact condition, we get that
\[
h=u(\tilde x)-\frac{A}{1-\alpha_{p}}|\tilde x-x_0|^{1-\alpha_{p}}+\frac{B}{2}|\tilde x-x_0|^2,
\]
and  $u(x_0)<v(x_0)$. This two facts together yield,
\[
u(x_0)<u(\tilde x)-\frac{A}{1-\alpha_{p}}|\tilde x-x_0|^{1-\alpha_{p}}+\frac{B}{2}|\tilde x-x_0|^2.
\]
Rearranging terms we get,
\[
A<(1-\alpha_{p})\frac{\big(u(\tilde x)-u(x_0)\big)}{|\tilde x-x_0|^{1-\alpha_{p}}}+\frac{B}{2}(1-\alpha_{p})|\tilde x-x_0|^{1+\alpha_{p}}.
\]
Notice that,
\[
\frac{B}{2}(1-\alpha_{p})|\tilde x-x_0|^{1+\alpha_{p}}<B\cdot{\rm diam}(\Omega)^{1+\alpha_{p}},
\]
so we arrive at a contradiction with the definition of $A$ and hence $\tilde x\notin\partial\Omega$.

\medskip

\noindent2.\quad{}Now, assume for contradiction that $\tilde x\neq x_0$. We have that $v$ touches $u$ from above at $\tilde x$. From the previous step we know that $\tilde x$ must be an interior point; hence from the hypothesis of contradiction we know that $v$ is $\mathcal{C}^2$ in a neighborhood of $\tilde x $. As $u$ solves \eqref{Dir} in the viscosity sense, we can use $v$ as a test function in the definition of viscosity solution and get
\[
-\Delta_p^{N} v(\tilde x)\leq f(\tilde x).
\]
On the other hand,  Lemma \ref{gradient.non.zero} implies  $\nabla v \ne 0 $ and then, from Lemma \ref{fs},
\[
(1+\varepsilon)\,\|f\|_{L^\infty(\Omega)}=  C_p^{-1} B= -\Delta_p^{N} v(\tilde x),
\]
with $\varepsilon>0$, a contradiction.
\end{proof}

We now complete the proof of Theorem \ref{uniform}.

\begin{proof}[Proof of Theorem \ref{uniform}]
Consider a cusp pointing downwards centered at $x$,
\[
v(y)=h+\frac{A}{1-\alpha_{p}}|y-x|^{1-\alpha_{p}}-\frac{B}{2}|y-x|^2
\]
with 
\[
A=(1-\alpha_p)\sup_{z\in\partial\Omega}
\frac{|u(z)-u(x)|}{|z-x|^{1-\alpha_{p}}}+
B\,{\rm diam}(\Omega)^{1+\alpha_{p}}
\]
and $B=(1+\varepsilon)\,C_p \|f\|_{L^\infty(\Omega)}$ with $\varepsilon>0$.
Lemma \ref{fundam} says that if $v$ touches the graph of $u$ from above, then the contact point is  $x$ and $h=u(x)$. We deduce,
\[
\begin{split}
u(y)&\leq v(y)=u(x)+\frac{A}{1-\alpha_{p}}|y-x|^{1-\alpha_{p}}-\frac{B}{2}|y-x|^2\\
&\leq u(x)+\frac{A}{1-\alpha_{p}}|y-x|^{1-\alpha_{p}}.
\end{split}
\]
We get,
\[
\begin{split}
\frac{u(y)-u(x)}{|y-x|^{1-\alpha_{p}}}\leq 
\sup_{z\in\partial\Omega}\,&
\frac{|u(z)-u(x)|}{|z-x|^{1-\alpha_{p}}}\\
&+
(1+\varepsilon)\,\|f\|_{L^\infty(\Omega)} 
\frac{C_p}{1-\alpha_p}
\,{\rm diam}(\Omega)^{1+\alpha_{p}}.
\end{split}
\]
As this estimate holds for any $\varepsilon>0$, we can let $\varepsilon\to0$.

On the other hand, we can do the same with a cusp pointing upwards and slide it until it touches from below. Namely, we get
\[
\begin{split}
u(y)&\geq v(y)=u(x)-\frac{A}{1-\alpha_{p}}|y-x|^{1-\alpha_{p}}+\frac{B}{2}|y-x|^2\\
&\geq u(x)-\frac{A}{1-\alpha_{p}}|y-x|^{1-\alpha_{p}}
\end{split}
\]
and then argue as before. The two estimates together yield the result.
\end{proof}

\bigskip

Now, we turn to the proof of the global estimates for the Dirichlet problem in Theorem \ref{uniform2}. A key point in the proof  is the following lemma, similar to  the comparison with cones property in \cite{Crandall.Evans.Gariepy}. 

\begin{lemma}\label{comp}
Consider $A>B\,{\rm diam}(\Omega)^{1+\alpha_{p}}$, $B=(1+\varepsilon)\,C_{p} \|f\|_{L^\infty(\Omega)}$ with $\varepsilon>0$,  and $x_0 \in \overline{\Omega}$. Let $u$ be a viscosity solution of \eqref{Dir}. Then if 
\[
u(x)\leq u(x_0)+A |x-x_0|^{1-\alpha_{p}}-\frac B 2|x-x_0|^2
\]
for all $x$ in $\partial \Omega$ the same inequality holds in the interior, that is 
\[
u(x)\leq u(x_0)+A |x-x_0|^{1-\alpha_{p}}-\frac B 2 |x-x_0|^2
\]
for all $x$ in  $\Omega$. In the same way, if 
\[
u(x)\geq u(x_0)-A |x-x_0|^{1-\alpha_{p}}+\frac B 2|x-x_0|^2
\]
for all  $x$ in $\partial \Omega$ the same inequality holds in the interior, that is 
\[
u(x)\geq  u(x_0)-A |x-x_0|^{1-\alpha_{p}}+\frac B 2 |x-x_0|^2
\]
for all  $x$ in  $\Omega$.
\end{lemma}

\begin{proof}We prove just the first claim since the other one is analogous.
Denote as usual
\[
v(x)= u(x_0)+A |x-x_0|^{1-\alpha_{p}}-\frac B 2|x-x_0|^2.
\]
As we intend to prove that $u(x)-v(x)\leq0$ for all $x\in\Omega$, assume for the sake of contradiction that $u(\tilde x)-v(\tilde x)=\max_{\Omega}(u-v)>0$ as we would be done otherwise. Since $u-v\leq0$ on $\partial\Omega$ by hypothesis, $\tilde x$ must be an interior point. Moreover, $u(\tilde x)-v(\tilde x)>0$ so $\tilde x \ne x_0$ and $v$ is $\mathcal{C}^2$ in a neighborhood of $\tilde x $. As $u$ is a viscosity solution of \eqref{Dir}, we have by definition that
\[
-\Delta_p^{N} v(\tilde x)\leq f(\tilde x)
\]
On the other hand, as $\nabla v \ne 0 $ by the hypothesis on the size of $A$ (see Lemma \ref{gradient.non.zero}) we have by Lemma \ref{fs} that
\[
-\Delta_p^{N} v(\tilde x)= C_p^{-1} B=(1+\varepsilon)\,\|f\|_{L^\infty(\Omega)}
\]
a contradiction, as $\varepsilon>0$.
\end{proof}

\begin{proof}[Proof of Theorem \ref{uniform2}]
Pick $x_0 \in \partial \Omega$.
Since $g \in\mathcal{C}^{0,1-\alpha_p}(\partial\Omega)$,  we have that for any $x\in \partial \Omega $
\[
u(x)=g(x)\le u(x_0)+L\,|x-x_0|^{1-\alpha_{p}}
\] 
for $L=[g]_{1-\alpha_p,\partial\Omega}$.
Let
\begin{equation}\label{ledzeppelin}
A\geq L+B\,{\rm diam}(\Omega)^{1+\alpha_{p}}
\end{equation}
and $B=(1+\varepsilon)\,C_{p} \|f\|_{L^\infty(\Omega)}$  with $\varepsilon>0$.

 It is easy to see that \eqref{ledzeppelin} implies,
\[
u(x)\le u(x_0)+L\,|x-x_0|^{1-\alpha_{p}}\le u(x_0)+A|x-x_0|^{1-\alpha_{p}}-\frac B 2 |x-x_0|^2
\] 
for any $x\in \partial \Omega$. Then,  Lemma \ref{comp} implies that the same holds  for any $x \in \Omega$.

With the same proof it also holds that 
\begin{equation}\label{Doors.Rockmebaby}
u(x)\ge u(x_0)-A|x-x_0|^{1-\alpha_{p}}+ \frac B 2 |x-x_0|^2
\end{equation}
for any $x \in \Omega$ and $x_0 \in \partial \Omega$.

Choose now $y \in \Omega$, by \eqref{Doors.Rockmebaby} we know that for any $x \in \partial \Omega$ we have 
\[
u(x)\le u(y)+A|x-y|^{1-\alpha_{p}}-\frac B 2 |x-y|^2
\]
and thanks  to Lemma \ref{comp} the same holds also for any $x \in \Omega$.
Reversing the role of $x$ and $y$ we obtain
\[
|u(x)-u(y)|\le A|x-y|^{1-\alpha_{p}}-\frac B 2 |x-y|^2\le A|x-y|^{1-\alpha_{p}}
\]
and the result follows.
\end{proof}

\medskip

\section{Limit equation}\label{section.limits}

\begin{theorem}\label{proposition.limit.eq}
Let $\Omega$ be a bounded domain,  $2\leq n<p\leq\infty$, and $u_{p}$ a viscosity solution of
\begin{equation}\label{p.finite}
\left\{
\begin{split}
-&\Delta_{p}^{N}u_{p}(x)=f_p(x)\quad\text{in}\ \Omega\\
&u_{p}(x)=g_p(x) \quad\text{on}\ \partial\Omega,
\end{split}
\right.
\end{equation}
with $f_p\in\mathcal{C}(\Omega)$ and $g_p\in\mathcal{C}^{0,\frac{p-n}{p-1}}(\partial\Omega)$. Suppose that  $f_p$ and $g_p$  converge uniformly to some $f\in\mathcal{C}(\Omega)$ and $g$ (which in turn is $\mathcal{C}^{0,1}(\partial\Omega)$). Then, there exists a subsequence
 $u_{p'}$  that converge uniformly to some $u$,  a viscosity solution of
\begin{equation}\label{limit}
\left\{
\begin{split}
-&\Delta_{\infty}^{N}u(x)=f(x)\quad\text{in}\ \Omega\\
&u(x)=g(x) \quad\text{on}\ \partial\Omega.
\end{split}
\right.
\end{equation}
Moreover, $u\in\mathcal{C}^{0,1}(\Omega)$.
\end{theorem}

As it was mentioned in the introduction, this result is related to the results in \cite{Bha-DiBe-Man}. However, a distinct feature of this limit process is the lack of variational structure of problem \eqref{p.finite} that yields complications in the proof of the uniform convergence of the solutions. Our ABP estimate provides a stable $L^\infty$ bound that can be used in combination with our stable regularity results to prove convergence.

It is also interesting to understand the limits of \eqref{p.finite} to solutions of \eqref{limit} as both equations can be interpreted in the framework of game theory, see \cite{PSSW, PSSW2}.


Once uniform estimates have been established, the proof of Theorem \ref{proposition.limit.eq} is rather standard, but we include it for the sake of completeness.

\begin{proof}
Fix $p_0$ such that $n<p_0$. Then, for $p\geq p_0$, 
Corollary \ref{C.alpha.estimates} yields the following estimate
\begin{multline*}
\|u_p\|_{\mathcal{C}^{0,\frac{p_0-n}{p_0-1}}(\Omega)}\leq {\rm diam}(\Omega)^{\left(\frac{p-n}{p-1}-\frac{p_0-n}{p_0-1}\right)}\cdot \|u_p\|_{\mathcal{C}^{0,\frac{p-n}{p-1}}(\Omega)}\\
\leq{\rm diam}(\Omega)^{\left(\frac{p-n}{p-1}-\frac{p_0-n}{p_0-1}\right)}\cdot\|g_p\|_{\mathcal{C}^{0,\frac{p-n}{p-1}}(\partial\Omega)}
\\ +{\rm diam}(\Omega)^{2-\frac{p_0-n}{p_0-1}}\cdot\left( \frac{2\,p}{p-1}\,{\rm diam}(\Omega)^{\frac{p-n}{p-1}}+\frac{p}{p+n-2} \right)\,
 \|f_p\|_{L^\infty(\Omega)}.
 \end{multline*}
As the right hand side can be bounded independently of $p$, Arzela-Ascoli Theorem yields the existence of a subsequence converging uniformly to some limit $u\in\mathcal{C}(\Omega)$. We will still denote the subsequence by $u_p$.

Now, we turn to checking that the limit $u$ is a viscosity solution of \eqref{limit}.
Let  $x_0\in\Omega$  and a function
$\varphi\in\mathcal{C}^{2}(\Omega)$ such that ${u}-\varphi$ attains a  local  minimum at $x_0$.  Up to replacing $\varphi(x)$ with $\varphi(x)-|x-x_0|^4$, we can assume without loss of generality that minimum to be strict.

As ${u}$  is the uniform limit of the subsequence $u_{p}$ and $x_0$ is a strict minimum point, there exists a sequence of points  $x_p\rightarrow{}x_0$ as $p\to\infty$ such that
$(u_{p}-\varphi)(x_p)$ is a local minimum for each $p$ in the sequence.

Assume first that $|\nabla \varphi(x_0)|>0$. Then, $|\nabla \varphi(x_p)|>0$ for $p$ large enough and, as $u_{p}$ is a  viscosity supersolution of \eqref{p.finite}, we have that,
\begin{equation}\label{realmadrid}
\begin{split}
-\frac{1}{p}\,\textnormal{trace}\left[\left(I+(p-2)\frac{\nabla  \varphi(x_p)\otimes\nabla  \varphi(x_p)}{|\nabla  \varphi(x_p)|^2}\right)D^2 \varphi(x_p)\right]&=-\Delta_p^N \varphi(x_p)\\
&\geq f_{p}(x_p).
\end{split}
\end{equation}
Letting $p\to\infty$ we get
\[
\begin{split}
-\Big\langle D^2\varphi(x_0)\frac{\nabla \varphi(x_0)}{|\nabla \varphi(x_0)|},\frac{\nabla \varphi(x_0)}{|\nabla \varphi(x_0)|}\Big\rangle=-\Delta_\infty^N \varphi(x_0)\geq f(x_0).
\end{split}
\]

If we assume otherwise that  $\nabla \varphi(x_0)=0$, we have to consider two cases. Suppose first that there exists a subsequence still indexed by $p$ such that $|\nabla\varphi(x_p)|>0$ for all $p$ in the subsequence. Then, by  Definition \ref{def.visc.inflapnorm}, we can  let $p\to\infty$ to  get
\[
\begin{split}
-\liminf_{p\to\infty}\Big\langle D^2\varphi(x_p)\frac{\nabla \varphi(x_p)}{|\nabla \varphi(x_p)|},\frac{\nabla \varphi(x_p)}{|\nabla \varphi(x_p)|}\Big\rangle&=-m\big(D^2\varphi(x_0)\big)\\
&=-\Delta_\infty^N \varphi(x_0)\geq f(x_0).
\end{split}
\]
If such a subsequence does not exists, according to Definition \ref{def.visc.plapnorm}, we have that
\[
-\frac1p\,\Delta\varphi(x_p)-\frac{(p-2)}{p}\,m\big(D^2\varphi(x_p)\big)=-\Delta_p^N\varphi(x_p)\geq f_p(x_p)
\]
for every $p$ large enough. Letting $p\to\infty$, we have 
\[
-m\big(D^2\varphi(x_0)\big)=-\Delta_\infty^N \varphi(x_0)\geq f(x_0).
\]
We have proved \eqref{def.supersol.inflap}, so $u$ is a viscosity supersolution of \eqref{limit}. The subsolution case is similar.

Finally, we conclude that $u\in\mathcal{C}^{0,1}(\Omega)$ either letting $p_0\to\infty$ or using  that $u$ is a solution of \eqref{limit}, and hence the estimates in Corollary  \ref{C.alpha.estimates} apply.
\end{proof}

\medskip

\section{Examples}\label{section.examples}

In this section we provide some examples. The first one shows that the classical ABP estimate in the form
\[
\sup_{\Omega} u \le \sup_{\partial \Omega} u+ C(n,\Omega)\|\Delta_\infty^N u\|_{L^n(\Omega)}
\]
fails to hold. In order to prove this we shall construct a sequence of functions $\{u_\varepsilon\}$ defined on $B_1(0)$, vanishing on the boundary such that 
\[
\sup_{B_1(0)} u_\varepsilon (x)\approx 1 \quad \text{and} \quad \|\Delta _\infty^N u_{\varepsilon}\|_{L^n(B_
1(0))} \to 0,
\]
as $\varepsilon\to 0$. To do this we define 
\[
u_\varepsilon (x) =v_\varepsilon (|x|)=\frac{1}{1+\varepsilon}\,\big(1-\ |x|^{1+\varepsilon}\big),
\] 
so that, 
\[
\sup_{B_1(0)} u_\varepsilon=\frac{1}{1+\varepsilon}
\]
and 
\begin{equation}\label{lapsbagliato}
-\Delta_\infty^N u_\varepsilon(x)=-v''_\varepsilon(|x|)=\frac{\varepsilon}{|x|^{1-\varepsilon}}.
\end{equation}
Now a straightforward computation gives
\[
\|\Delta^N_\infty u_\varepsilon \|^n_{L^n(B_1(0))}=c(n) \varepsilon^n \int_0^1 \frac{d\rho}{\rho^{1-n\varepsilon}}=c(n)\varepsilon^{n-1}.
\]

Notice however that  equation \eqref{lapsbagliato} is a pointwise computation and one has to give it a meaning in the viscosity sense. However this difficulty can  be easily overcome considering, for instance, the functions
\[
u_{\varepsilon, \delta}(|x|)=\frac{(1+\delta^2)^{\frac{1+\varepsilon}{2}}}{1+\varepsilon}-\frac{\ (|x|^2+\delta^2)^{\frac{1+\varepsilon}{2}}}{1+\varepsilon} 
\]
which  are $\mathcal{C}^2$, details are left to the reader.

As we said in the introduction we are not able to provide a counterexample to the validity of an estimate of the form
\begin{equation}\label{ABPP}
\sup_{\Omega} u \le \sup_{\partial \Omega} u + C(n,p,\Omega)\,\|\Delta_\infty^N u\|_{L^p(\Omega)}
\end{equation}
for $p>n$. Our strategy, which is to approximate the ``infinity harmonic" cone $1-|x|$ with radial functions, cannot  be used to obtain such a counterexample. In fact, it is easy to see that if $v_\varepsilon(|x|)$ uniformly converge to $1-|x|$, then the sequence of unidimensional functions $v''_\varepsilon(\rho)$  has to converge 
 in the sense of distributions to $-2\delta_0$, twice a Dirac mass in the origin. Hence to find a counterexample to \eqref{ABPP} with our strategy one essentially needs to construct a sequence $v_\varepsilon(\rho)$ such that 
\[
 \int_0^1 v''_\varepsilon(\rho)\,d\rho \to 1 \quad \text{and} \quad \int_0^1 (v''_\varepsilon(\rho))^p \rho^{n-1}\,d\rho \to 0.
\]
An application of H\"older inequality gives that this is impossible  for $p>n$. 

The second example we provide shows that our estimate is, in some sense, sharper than the plain $L^\infty$ bound
\begin{equation}\label{ABPINF}
\sup_{\Omega} u \le \sup_{\partial \Omega} u + c(n,\Omega)\|\Delta_\infty^N u\|_{L^\infty (\Omega)}.
\end{equation}
Let us consider the functions 
\[
w_\varepsilon(x)=
\begin{cases}
1-|x| \qquad &\text{if $\varepsilon< |x|\le 1$}\\
1-\frac{\varepsilon}{2}-\frac{|x|^2}{2\varepsilon.} &\text{if $|x|\le \varepsilon$}.
\end{cases}
\]
One immediately sees that
\[
\sup_{B_1(0)} w_\varepsilon\approx 1\quad \text{and} \quad \|\Delta_\infty^N w_\varepsilon\|_{L^\infty (B_1(0))}=\frac{1}{\varepsilon},
\]
while
\[
\int_0^{\sup_{B_1(0)} w_\varepsilon} \|\Delta_\infty^N w_\varepsilon\|_{L^\infty(\{w_\varepsilon=r\})} dr\approx 1,
\]
so our estimate can give more information than \eqref{ABPINF}.

\appendix

\section{}

For the sake of completeness we provide here the proof of some technical details needed in the foregoing.

\begin{lemma}\label{lemma.cici}
Let $\Omega^*$ be a convex domain and $u$  a continuous semiconvex function  such that $u\leq0$ in $\mathbb{R}^n\setminus\Omega^*$ and $u(x_0)>0$ for some $x_0 \in \Omega^*$. Let $\Gamma_\sigma (u^+)$ be the concave envelope of $u^+$ (extended by 0 outside $\Omega^*$) in 
\[
\Omega^*_\sigma=\{x\in\mathbb{R}^n  \textrm{ such that } {\rm dist}(x,\Omega^*)\le\sigma\}.
\]
Then $\Gamma_\sigma(u)$ is $\mathcal C_{\rm loc}^{1,1}$ in $\Omega_\sigma^*$ (and hence second order differentiable a.e.). Moreover,
\[
\{x\in \Omega^*:\ \det D^2 \Gamma_\sigma (u)\neq0\}\subset C_\sigma ^{+}(u),
\]
where $C_\sigma ^+(u)$ is the set of points in $\Omega^*_\sigma$ where $u=\Gamma_\sigma(u)$.
\end{lemma}

\begin{proof}
This lemma can be easily deduced from the arguments of \cite[Lemma 3.5]{CC}. However, since our statement is slightly different,  we shall briefly sketch the proof for the sake of completeness. Obviously the second derivative of $\Gamma_\sigma (u)$ is bounded from above by $0$ and hence to prove the $\mathcal C_{\rm loc}^{1,1}$ regularity it is enough to show they are locally  bounded from below. Since the graph of $u$ can be touched from below by a paraboloid with fixed opening this is clearly true for every point in the contact set $C_\sigma^+ (u)$.

Let now choose a compact set $\Omega^* \subset K\subset \Omega^*_\sigma$, pick  $x\in K \setminus C_\sigma ^+(u) $ and let $L$ be a supporting hyperplane at $x$ to the graph of $\Gamma_\sigma (u^+)$. Following the same argument of \cite[Lemma 3.5]{CC} one can prove that
\[
x\in {\rm conv}\{y \in  \Omega^*_\sigma \textrm{ such that } u^+(y) =L(y)\},
\]
where $\rm conv$ denotes the closed convex hull of a set.
Moreover,
 by Caratheodory's Theorem, we can  write 
\[
x=\lambda_1 x_1+\dots+\lambda_{n+1} x_{n+1}
\]
with $x_i \in \{ u^+=L\} $, $\lambda_i \ge 0$ and $\sum_{i=1}^{n+1} \lambda_i =1$. 
We conclude that $x$ belongs to the simplex ${\rm conv}\{x_1,\ldots,x_{n+1}\}$
and that
$L\equiv\Gamma_\sigma (u^+)$ in this set. Moreover, all the $x_i\in\Omega^*\cap C_\sigma^+(u)$  with the exception of at most one, that may belong  to  $\partial \Omega_\sigma^* $; otherwise by concavity one would have that $\Gamma_\sigma (u^+)=0$ everywhere in $\Omega^*_\sigma$ since $u^+=0$ in $\Omega^*_\sigma \setminus \Omega^*$.

We now claim that there exists a constant $C=C(\sigma,n,\Omega^*, K)$ and an index $i_0$ such that  $x_{i_0}\in \Omega ^*$ and $\lambda_{i_0} \ge C$, this is  true  with $C=1/(n+1)$ if all the $x_i$ belong to $\Omega^*$. Otherwise we can assume without loss of generality that $x_{n+1} \in \partial \Omega_\sigma^*$ and  suppose that $\lambda_i < C $ for $i=1,\dots,n$. Then,
\[
0<{\rm dist}(K,\partial \Omega^*_ \sigma) \le |x-x_{n+1}|\le\sum_{i=1}^{n}\lambda_i\left|x_i-x_{n+1}\right|<  n\,C( {\rm diam}(\Omega^*)+2 \sigma ),
\]    
 a contradiction if $C$ is  small enough. Assume without loss of generality that $i_0=1$, by the semiconvexity assumption we know that
\[
\Gamma_\sigma (u^+)(x_1+h)\ge u^+(x_1+h)\ge L(x_1+h) -M|h|^2
\]
for $|h|\le {\rm dist} (\Omega^*, \partial \Omega_\sigma^*)$ and $L$ an affine function whose graph is the supporting hyperplane to the graph of $\Gamma_\sigma(u^+)$ at $x_1$. Writing
\[
x+h=\lambda_1 \Big(x_1+\frac{h}{\lambda_1}\Big)+\dots+\lambda_{n+1} x_{n+1},
\]
thanks to the concavity of $\Gamma_\sigma (u^+)$ and the fact that $L\equiv\Gamma_\sigma (u^+)$ in the simplex ${\rm conv}\{x_1,\ldots,x_{n+1}\}$, we get
\[
\begin{split}
\Gamma_\sigma (u^+)(x+h)&\ge \lambda_1\Gamma_\sigma (u^+) \Big(x_1+\frac{h}{\lambda_1}\Big)\\
&\quad+\lambda_ 2 \Gamma_\sigma(u^+)(x_2)+\dots+\lambda_{n+1} \Gamma_\sigma (u^+)(x_{n+1})\\
&\ge \lambda_1 L\Big(x_1+\frac{h}{\lambda_1}\Big) -\frac {M} {\lambda_1}|h|^2+\lambda_ 2 L(x_2)+\dots+\lambda_{n+1} L(x_{n+1})\\
&=L(x+h)-\frac {M} {\lambda_1}|h|^2\ge L(x+h)-\tilde M (n,\sigma, \Omega^*,K)|h|^2
\end{split}
\]
for $|h|\le \varepsilon (n,\sigma, \Omega^*,K)$ small enough, which is the claimed $\mathcal{C}^{1,1}_{\rm loc }$ regularity. To prove the last assertion in the statement of the Lemma, just notice that 
 for any $x \in \Omega^{*} \setminus C^+_\sigma (u)$ the function $\Gamma_\sigma (u^+)$ coincides with an affine function on a segment.
\end{proof}

\begin{lemma}\label{pointwise}
Let $u$ be a viscosity subsolution of
\[
\mathcal{F}\big(\nabla u,D^{2}u\big)= f(x)\quad\text{in}\ \Omega
\]
 and let $x_0\in \Omega$. If $\varphi$ is a function such that $u-\varphi$ has a local maximum at $x_0$ and  $\varphi$ is twice differentiable at $x_0$, then
\[
\mathcal{F}_*\big(\nabla \varphi({x_{0}}),D^{2}\varphi({x_{0}})\big)\leq f({x_{0}}).
\]
\end{lemma}

\begin{proof}
By assumption
\[
\varphi(x)=\varphi(x_0)+\langle\nabla \varphi(x_0), (x-x_0)\rangle+ \frac 1 2 \langle D^2\varphi(x_0)  (x-x_0),(x-x_0)\rangle+o(|x-x_0|^2).
\]
Thus for every $\varepsilon>0$ the  $\mathcal C^2$ function
\[
\varphi_\varepsilon(x)=\varphi(x_0)+\langle\nabla \varphi(x_0), (x-x_0)\rangle+ \frac 1 2 \langle D^2\varphi(x_0)  (x-x_0),(x-x_0)\rangle+\frac{\varepsilon}{2}|x-x_0|^2.
\]
touches $u$ from above at $x_0$, so by definition of viscosity subsolution we have
\[
\mathcal{F}_*\big(\nabla \varphi({x_{0}}),D^{2}\varphi({x_{0}})+\varepsilon{\rm Id}\big)\leq f({x_{0}})
\]
and passing to the limit in $\varepsilon$ we obtain the thesis.
\end{proof}

Before stating the next  Proposition, we recall that a Lipschitz function $v: M \to N$  between two $\mathcal{C}^1$ manifolds is differentiable $\mathcal{H}^{n-1}$-a.e. and we shall denote with $\nabla^M v$  its tangential gradient, that is, the linear operator between $T_x M$ and $T_x N$ defined by
\[
\nabla^M v\cdot \xi:=\lim_{t\to0}\frac{f(x+t\xi)-f(x)}{t},
\]
see for instance \cite[Section 3.1]{Federer}.

\begin{proposition} \label{proposition.hardcore}
Let $w$ be a concave  $\mathcal{C}^{1,1}$  function in a ball $ B_{R}(0) $ with $w=0$ on $\partial  B_{R}(0) $. Then,
\begin{enumerate}

\item For every level $r\in(0,\sup_{ B_{R}(0) }w)$, $\nabla w\neq0$ and $M_r:=\{w=r\}$ is a $\mathcal{C}^{1,1}$ manifold.

\item For a.e. $r\in[0,\sup_{ B_{R}(0) }w]$, $\mathcal{H}^{n-1}$-a.e. $x\in M_r$ is a point of twice differentiability for $w$.

\item Let $r\in(0,\sup_{ B_{R}(0) } w)$, if $x\in M_r$ is a point of twice differentiability, then,
\begin{equation}\label{tanghessian}
\begin{split}
-D^2w (x) =&|\nabla w(x)|\,\sum_{i=1}^{n-1} \kappa_i(x)\, \tau_i(x)\otimes \tau_i(x)\\
&-\sum_{i=1}^{n-1}\partial_{\nu\tau_{i}}^2w(x)\, \big(\tau_i(x)\otimes\nu(x)+\nu(x)\otimes\tau_i(x)\big)\\
&-\Delta_\infty^N w(x)\,  \nu(x) \otimes \nu(x),
\end{split}
\end{equation}
where:
\begin{enumerate}
\item $\nu(x)$ is the exterior normal to $M_r$ at $x$, i.e.
\[
\nu(x)=-\frac{\nabla w(x)}{|\nabla w(x)|}.
\]
\item $\tau_i(x)$ for $i=1,\ldots, n-1$ is an orthonormal basis of $T_xM_r$ that diagonalizes the Weingarten operator $\nabla^M\nu$ at $x$.
\item $\kappa_i(x)$ are the eigenvalues of $\nabla^M \nu$ at $x$, which are the principal curvatures of $M_r$ at $x$.
\end{enumerate}
\end{enumerate}
\end{proposition}

\begin{proof}
\noindent(1)\quad{}As $w$ is concave we have that $\nabla w(x)=0$ if and only if $x$ is such that $w(x)=\sup_{ B_{R}(0) }w$. Then, by the Implicit Function theorem, for every $r\in(0,\sup_{ B_{R}(0) } w)$ the level set $M_r$ can be locally represented as a graph of a $\mathcal{C}^{1,1}$ function. Hence, up to a change of coordinates, we can suppose that there exists $f:\mathcal{U}\subset \mathbb{R}^{n-1}\to\mathbb{R}$ such that 
\[
w(x',f(x'))=r,\qquad\text{and}\quad \partial_{n} w(x',f(x'))\geq c>0\quad\forall x'\in \mathcal{U},
\]
where  we  denote $x=(x',x_n)\in\mathbb{R}^{n-1}\times \mathbb{R}$.  In particular, 
\[
\nabla' f(x')=-\frac{\nabla'w(x',f(x'))}{\partial_{n} w(x',f(x'))}
\]
is Lipschitz.

\noindent(2)\quad{} Now, let $D$ be the set of points of twice differentiability for $w$. Then, we know that $| B_{R}(0) \setminus D|=0$ (see \cite{Evans-Gariepy}) and by the Coarea Formula,
\[
0=\int_{ B_{R}(0) \setminus D}|\nabla w(x)|\,dx=\int_0^{\sup_{ B_{R}(0) } w} \mathcal{H}^{n-1}\big(M_r\cap( B_{R}(0) \setminus D)\big)\,dr.
\]

\noindent(3)\quad{} Let now $x_0$ be a point of twice differentiability for $w$. By part (1), we know that $\nabla w(x_0)\neq0$. Without loss of generality we can assume $x_0=0$.  By means of an axis rotation, we can also suppose that 
 $\nu(0)=e_{n}$, that is,
\[
\nabla w(0)=\partial_{n}w(0,0)\,e_n
\]
with $\partial_{n}w(0,0)<0$.

There exists $f:B_\delta'(0)\subset \mathbb{R}^{n-1}\to\mathbb{R}$ such that $w(x',f(x'))=r$, in particular, $f(0)=0$ and $\nabla'f(0)=0$. Moreover,
\[
\nabla'f(x')=-\frac{\nabla'w(x',f(x'))}{\partial_{n}w(x',f(x'))}
\]
is differentiable at 0.

Using another rotation of axis affecting only the first 
$n-1$ variables, we can suppose that the matrix  $D_{x'}^2  f(0)$ is diagonal 
in our coordinates, namely,
\begin{equation}\label{guido.says.a.word}
D_{x'}^2  f(0)={\rm diag}(-\kappa_1,\ldots,-\kappa_{n-1}).
\end{equation}

Moreover, the Gauss map coincides with
\[
x'\mapsto \frac{1}{\sqrt{1+|\nabla'f(x')|^2}}\,\big(-\nabla'f(x'),1\big),
\]
so differentiating, we obtain
\[
\nabla^{M}\nu(0)=-D_{x'}^2f(0).
\]
In particular, from this equality and \eqref{guido.says.a.word}, we deduce  that  $\tau_1(0),\ldots,\tau_{n-1}(0)$ and $e_1,\ldots,e_{n-1}$ coincide.

Differentiating twice the expression $w(x',f(x'))=r$ with respect to $x'$ we finally obtain,
\[
\begin{split}
D_{x'}^2w(0,0)&=-\partial_{n}w(0,0)\, D_{x'}^2f(0)\\
&=|\nabla w(0)|\, D_{x'}^2f(0)=-|\nabla w(0)|\, \nabla^{M}\nu(0).
\end{split}
\]
Moreover, we also have $\partial_{\nu,\nu}^2w(0)=\langle D^2w(0)\nu(0),\nu(0)\rangle=\Delta_\infty^{N}w(0)$ and the thesis is proved.
\end{proof}

We recall the following version of the Gauss-Bonnet Theorem, the proof of which easily follows from the Area Formula between rectifiable sets (see \cite[Corollary 3.2.20]{Federer}).

\begin{theorem}\label{GBthheorem}Let $K\subset\mathbb{R}^{n}$ a $\mathcal{C}^{1,1}$ convex set and $\nu_K$ its outer normal. Then,
\begin{equation}\label{GBeq}
\int_{\partial K} \prod_{i=1}^{n-1} \kappa_i(x)\, d\mathcal H^{n-1}=\int_{\partial K}\det( \nabla^{\partial K} \nu_K)\,d\mathcal H^{n-1}=\mathcal H^{n-1}\big(\partial B_1(0)\big). 
\end{equation}
\end{theorem}

\begin{proof} By the  Area Formula (see \cite[Corollary 3.2.20]{Federer}) we have
\[
\begin{split}
\int_{\partial K}\det( \nabla^{\partial K} \nu_K)\,d\mathcal H^{n-1}&=\int_{\partial B_1(0)} {\rm deg}(\nu_K,\partial K, \partial B_1(0))\,d \mathcal H^{n-1}\\
&=\mathcal H^{n-1}\big(\partial B_1(0)\big), 
\end{split}
\]
where  ${\rm deg}(\nu_K,\partial K, \partial B_1(0))$ is the Brouwer degree of the Gauss map $\nu_{K}:\partial K\to\partial B_1(0)$ which can be easily seen to be 1.
\end{proof}

\

{{\bf\noindent Acknowledgements:} This paper was started while A. Di Castro and G. De Philippis were visiting the Mathematics Department of the University of Texas at Austin for whose hospitality the authors are very grateful.

}

\


\bibliographystyle{plain}


\end{document}